\documentclass{amsart}
\usepackage{amssymb}
\usepackage{enumerate}
\usepackage{mathrsfs}
\usepackage{relsize}

\newtheorem{theorem}{Theorem}[section]

\newtheorem{lemma}[theorem]{Lemma}

\newtheorem{corollary}[theorem]{Corollary}

\theoremstyle{definition}
\newtheorem{definition}[theorem]{Definition}

\theoremstyle{remark}
\newtheorem{remark}[theorem]{Remark}

\numberwithin{equation}{section}

     \DeclareMathOperator{\Aut}{Aut}
     
     \DeclareMathOperator{\Inn}{Inn}

    \DeclareMathOperator{\Mod}{Mod}

    \DeclareMathOperator{\Ext}{Ext} \DeclareMathOperator{\II}{II}
    \DeclareMathOperator{\vN}{vN}

\def\R{{\mathbb R}}
\def\C{{\mathbb C}}

\def\N{{\mathbb N}}
\def\Z{{\mathbb Z}}

\def\F{{\mathbb F}}

\def\T{{\mathbb T}}

\begin{document}

\title[Classification of Gamma factors]{A note on the classification of Gamma factors}

\author{Rom\'an Sasyk}

\address{Departamento de Matem\'atica, Facultad de Ciencias Exactas y Naturales, Universidad de Buenos Aires, Argentina}
\address{
and}
\address{
Instituto Argentino de Matem\'aticas-CONICET\\
Saavedra 15, Piso 3 (1083), Buenos Aires, Argentina}

\email{rsasyk@dm.uba.ar}
\thanks{ The author acknowledges support from the following grants: PICT 2012-1292 (ANPCyT), and  UBACyT 2011-2014 (UBA)}

\subjclass[2000]{46L36; 03E15; 37A15}


\keywords{von Neumann algebras; descriptive set theory; Gamma
factors}

\maketitle

\begin{abstract}
One of the earliest invariants introduced in the study of finite von Neumann algebras is the property 
$\Gamma$ of Murray and von Neumann. 
In this note we prove that it is not possible to classify separable 
$\II_1$ factors satisfying the property $\Gamma$ up to isomorphism by a Borel measurable assignment of countable structures as invariants. We also show that the same holds true for the full $\II_1$ factors.
\end{abstract}

\section{Introduction}

In this note we continue with the line of research initiated by the author in collaboration with A. T\"ornquist  in \cite{sato09b},
\cite{sato09a} and \cite{sato09c}, where we applied the notion of
{\it Borel reducibility} from descriptive set theory to study the
complexity of the classification problem of several different
classes of separable von Neumann algebras. 

 Recall that if $E$ and $F$ are equivalence relations on standard
Borel spaces $X$ and $Y$, respectively, we say that $E$ is
\emph{Borel reducible} to $F$ if there is a Borel function $f:X\to
Y$ such that
$$
(\forall x,x'\in X) x E x'\iff f(x) F f(x'),
$$
and if this is the case we write $E\leq_B F$. Thus if $E\leq_B F$
then the points of $X$ can be classified up to $E$ equivalence by a
Borel assignment of invariants that we may think of as
$F$-equivalence classes. $E$ is {\it smooth} if it is Borel
reducible to the equality relation on $\R$. While smoothness is
desirable, it is most often too much to ask for. A more generous
class of invariants which seems natural to consider are countable
groups, graphs, fields, or other countable structures, considered up
to isomorphism. Thus, following \cite{hjorth00}, we will say that an
equivalence relation $E$ is {\it classifiable by countable
structures} if there is a countable language $\mathcal L$ such that
$E\leq_B\simeq^{\Mod(\mathcal L)}$, where $\simeq^{\Mod(\mathcal
L)}$ denotes isomorphism in $\Mod(\mathcal L)$, the Polish space of
countable models of $\mathcal L$ with universe $\N$.

In \cite{sato09a} it was proved that the isomorphism relation in the set
of finite von Neumann algebras is not classifiable by countable
structures. Nonetheless, it can certainly be the case that some
subclasses of finite factors are possible to classify by countable structures. For instance, Connes'
celebrated Theorem \cite{connesinj76}, says that the set of infinite
dimensional injective finite factors has only one element on its
isomorphism class, namely the hyperfinite $\II_1$ factor $R$. In
contrast with the injective case, in this note we show that it is
not possible to obtain a reasonable classification up to
isomorphisms for a well studied family of finite factors that
includes $R$. In order to state our results we observe first that
the set of finite factors can be split in two disjoint subsets:
those who satisfy the property $\Gamma$ of Murray and von Neumann
and those who are {\it full}. The first set contains the hyperfinite
$\II_1$ factor $R$ and more generally, the class of McDuff factors,
i.e. those factors of the form $M\otimes R$ for $M$ a $\II_1$
factor. On the other hand the set of full factors contains the free
group factors $L(\F_n)$. In this article we show that the $\II_1$ factors constructed in
\cite{sato09a} are full. As a consequence, Theorem  7 in \cite{sato09a} strengthens to prove:

\begin{theorem}\label{fullthm}
The isomorphism relation for full type $\II_1$ factors is not
classifiable by countable structures.
\end{theorem}

It remained then to analyze the complexity of the classification of $\II_1$ factors with the property $\Gamma$. In this note we address this problem by showing that:

\begin{theorem}\label{mainthm}
The isomorphism relation for McDuff  factors is not classifiable by
countable structures.
\end{theorem}
An immediate consequence is:
\begin{corollary}
The isomorphism relation for type $\II_1$ factors satisfying the
property $\Gamma$ of Murray and von Neumann is not classifiable by
countable structures.
\end{corollary}

We end this introduction by mentioning that the study of the connections between logic and operator algebras has recently attracted many researchers from both fields. As a consequence, in the past five years there has been a burst of activity in proving results along the lines of the ones presented in this note and first unveiled in \cite{sato09b}, \cite{sato09a} and \cite{sato09c}.  We refer the reader who wants to learn more on these exciting new developments to the recent survey of I. Farah  \cite{farah}.\\

\section{Gamma factors}

 We start by
recalling the definitions of the objects we study in this article.
Let $\mathcal H$ be an infinite dimensional separable complex
Hilbert space and denote by $\mathcal B(\mathcal H)$ the space of
bounded operators on $\mathcal H$, which we give the weak topology.
A separable von Neumann algebra is a weakly closed self-adjoint
subalgebra of $\mathcal B(\mathcal H)$. The set of von Neumann
algebras acting on $\mathcal H$ is denoted $\vN(\mathcal H)$. A von
Neumann algebra $M$ is said to be {\it finite} if it admits a finite
faithful normal tracial state, i.e. a linear functional $\tau:M\to
\C$ such that: $\tau(x^*x)\geq 0$, $\tau(x^*x)=0$ iff $x=0$,
$\tau(1)=1$\,, $\tau(xy)=\tau(yx)$ and the unit ball of $M$ is
complete with respect to the norm given by the trace
$\|x\|_2=\tau(x^*x)$. If a finite von Neumann algebra is also a {\it
factor}, i.e. its center is trivial, then it has a unique such a
trace. A finite von Neumann factor that is not a matrix algebra is
called a type $\II_1$ factor. This terminology is due to the general
classification of von Neumann algebras according to types (see
\cite[Chapter 5.1]{connesNCG} for an historical account of the
theory of types).

In this note we will be interested in $\II_1$ factors arising from
the so called {\it group-measure space construction}, that we
proceed to describe. For that,
 let $G$ be a
countably infinite discrete group which acts in a measure preserving
way on a Borel probability space $(X,\mu)$. For each $g\in G$ and
$\zeta\in L^2(X,\mu)$ the formula
$$
\sigma_g(\zeta)(x)=\zeta(g^{-1}\cdot x)
$$
defines a unitary operator on $L^2(X,\mu)$.

We identify the Hilbert space $\mathcal H=L^2(G,L^2(X,\mu))$ with
the Hilbert space of formal sums $\sum_{g\in G}\zeta_g\xi_g$, where
the coefficients $\zeta_g$ are in $L^2(X,\mu)$ and satisfy
$\sum_g\|\zeta_g\|_{L^2(X,\mu)}^2<\infty$, and $\xi_g$ are
indeterminates indexed by the elements of $G$. The inner product on
$\mathcal H$ is given by
$$
\langle \sum_{g\in G}\zeta_g(x)\xi_g, \sum_{g\in
G}\zeta'_g(x)\xi_g\rangle=\sum_{g\in G}
\langle\zeta_{g},\zeta'_{g}\rangle_{L^2(X,\mu)}.
$$
Both $L^\infty(X,\mu)$ and $G$ act by left multiplication on
$\mathcal H$ by the formulas
\begin{align*}
&f(\zeta_g(x)\xi_g)=((f(x)\zeta_g(x))\xi_g,\\
&u_h(\zeta_g(x)\xi_g)=\sigma_h(\zeta_g)(x)\xi_{hg},
\end{align*}
where $f\in L^\infty(X,\mu)$, $\zeta_g(x)\in L^2(X,\mu)$ and $g,h\in
G$. Thus if we denote by $\mathcal{FS}$ the set of finite sums,
$$
\mathcal{FS}=\{\sum_{g\in G}f_gu_g: \,f_g\in L^\infty(X,\mu),\,\,
f_g=0, \text{ except for finitely many } g\},
$$
then each element in $\mathcal{FS}$ defines a bounded operator on
$\mathcal H$. Moreover, multiplication and involution in
$\mathcal{FS}$ satisfy the formulas
$$
(f_gu_g)(f_hu_h)=f_g\sigma_g(f_h)u_{gh}
$$
and
$$
(fu_g)^*=\sigma_{g^{-1}}(f^*)u_{g^{-1}}
$$
and so $\mathcal{FS}$ is a $*$-algebra. By definition, the {\it
group-measure space von Neumann algebra} is the weak operator
closure of $\mathcal{FS}$ on $\mathcal B(\mathcal H)$ and it is
denoted by $L^\infty(X,\mu)\rtimes_\sigma G$. The trace on
$\mathcal{FS}$, defined by
$$
\tau( \sum_{g\in G}f_gu_g)=\int_X f_e\,d\mu,
$$
extends to a faithful normal tracial state in
$L^\infty(X)\rtimes_{\sigma} G$ by the formula $\tau(T)=\langle
T(\xi_e),\xi_e\rangle$, where $e$ represents the identity of $G$.

\begin{definition}[Murray-von Neumann \cite{MvN4}]
A finite von Neumann algebra $M$ has the property $\Gamma$ if given
$ x_1,\dots,x_n\in M,$ and $\varepsilon>0$\,there exists $u\in
\mathcal U(M),\tau(u)=0$ such that
$$\|x_iu-ux_i\|_2<\varepsilon, \text{ for all }\, 1\leq i\leq n.$$
\end{definition}

\vskip .1in

It follows immediately from its definition that the hyperfinite
$\II_1$ factor $R$ is a $\Gamma$-factor. Moreover, it is clear that
any finite factor of the form $M\otimes N$ with $N$ a
$\Gamma$-factor is also a $\Gamma$-factor. In particular, {\it
McDuff factors}, i.e. factors of the form $M\otimes R$, are
$\Gamma$-factors. The paradoxical decomposition of the free groups
$\F_n$, $n\geq 2$ is the key ingredient  \cite[Lemmas 6.2.1,
6.2.2]{MvN4} to show that the corresponding group von Neumann
factors $L(\F_n)$, $n\geq 2$ do not have the property $\Gamma$. More
generally, Effros showed in \cite{effros75} that if $G$ is a
discrete ICC\footnote{ICC stands for infinite conjugacy classes. $G$
is ICC if and only if $L(G)$ is a factor.} group and $L(G)$ has the
property $\Gamma$, then $G$ is inner amenable, (so in particular, free groups are not inner amenable). That the converse of Effros' theorem is
false is a recent result of Vaes \cite{vaes09}.\\

 If $M$ is finite von
Neumann algebra with trace $\tau$, $\Aut(M,\tau)$, the set of
$\tau$-preserving automorphisms of $M$ is a Polish group. A basis
for that topology is given by the sets $\mathcal
V_{T,a_1,\dots,a_n,\varepsilon}=\{S\in \Aut(M,\tau):
||S(a_i)-T(a_i)||_2\leq \varepsilon, \forall 1\leq i\leq
n,\,\,a_i\in M \}$. $\Inn(M)$ denotes the set of inner automorphisms
of $M$, i.e. those of the form $Ad(u)$, $u\in\mathcal U(M)$.

\begin{definition}[Connes \cite{connes74}]
A finite von Neumann algebra $M$ is full if $\Inn(M)$ is closed in
$\Aut(M)$.
\end{definition}

In \cite[Corollary 3.8]{connes74}, Connes showed that a $\II_1$
factor $M$ is full if and only if $M$ does not have the property
$\Gamma$.  It follows that for each $n\geq 2$, $L(\F_n)$ is a full
factor.
In order to discern when group measure space von Neumann algebras
are full we need the following:

\begin{definition}[Schmidt \cite{Schmidt80}]
Let $G$ be a discrete group and let $\sigma$ be an ergodic measure
preserving action of $G$ on a probability space $(X,\mu)$. A
sequence $\left ( B_n\right )_{n\in\N}$ of measurable subsets of $X$
is {\it asymptotically invariant}
 if
$$\mu(B_n\triangle\sigma_g(B_n))\to 0, \text{ for all }\, g\in G.$$
The sequence is {\it trivial} if
$$ \mu(B_n)(1-\mu(B_n))\to 0.$$
The action $\sigma$ is  {\it strongly ergodic} if every
asymptotically invariant sequence is trivial.
\end{definition}

The relation between strong ergodicity and fullness has been studied
by several authors. For the purpose of this note, it is enough to
mention the following Theorem of Choda \cite{choda82}:

\begin{theorem}\label{TheoChoda} Let $G$ be a
discrete group that is not inner amenable, and let $\sigma$ be a strongly ergodic measure presearving action
of $G$ on a probability space $(X,\mu)$. Then
$L^\infty(X,\mu)\rtimes_\sigma G$ is a full factor.
\end{theorem}

\begin{remark} The condition that is really used in the proof
of Theorem \ref{TheoChoda} is that $L(G)$ is full.
\end{remark}

It is known that a group is amenable if and only if it does not
admit strongly ergodic actions \cite{Schmidt80}, while a group has
the property (T) of Kazdhan if and only if every m.p. ergodic action
of it is strongly ergodic \cite{ConnesWeiss80}.
We describe now a concrete example of a strongly ergodic action of
$\F_2$ that we will use in this work. Since $\F_2$ can be identified
with the finite index subgroup of $SL(2,\Z)$ generated by the
matrices $\left\{
\begin{bmatrix} 1 & 0\\ 2& 1\end{bmatrix},
\begin{bmatrix} 1 & 2\\ 0& 1\end{bmatrix}\right\}$
 (see \cite[II.B.25]{delaHarpe}), it follows that $\F_2$ naturally
acts on $\T^2$. Lets denote such action by $\sigma$ and by $T_a$,
$T_b$ the automorphisms corresponding to the generators $a$, $b$ of
$\F_2$. This action is clearly measure preserving, and one of the
main results in \cite{Schmidt80} is that $\sigma$ is strongly
ergodic. Inspired by earlier work of Gaboriau and Popa
 and T\"ornquist (\cite{GaboriauPopa}, \cite{tornquist06}), in
\cite{sato09a} we used this action as the starting point for
showing that $\II_1$ factors are not classifiable by countable
structures. More precisely, the set
$$\Ext(\sigma)=\{S\in\Aut(\T^2,\mu): T_a,T_b,S \text{ generates a free action of
}\F_3\}$$ was shown in \cite[\S 3]{tornquist06} to be a dense
$G_\delta$ subset of $\Aut(\T^2,\mu)$. Thus $\Ext(\sigma)$ is a
standard Borel space. For each $S\in\Ext(\sigma)$ denote by
$\sigma_S$ the corresponding $\F_3$ action and $M_S\in
\vN(L^2(\F_3,L^2(\T^2,\mu))$ the corresponding group-measure space
von Neumann algebra
$$M_S=L^\infty(\T^2)\rtimes_{\sigma_{S}}\F_3.$$\\

In \cite{sato09a} the author and  T\"ornquist  showed:
\begin{theorem}\label{satomain}The equivalence relation on $\Ext(\sigma)$ given by
$S\simeq^{\mathscr F_{\II_1}} S'$ if $M_S$ is isomorphic to $M_{S'}$
is not classifiable by countable structures.
\end{theorem}

In \cite{sato09a} it was shown that $S\to M_S$ is a Borel map from $\Ext(\sigma)$ to
$\vN(L^2(\F_3,L^2(\T^2,\mu)))$. Thus
Theorem \ref{fullthm} is an immediate consequence of the previous theorem
and of the next:

\begin{lemma}\label{lemmaFull} For each $S\in Ext(\sigma)$, $M_S$ is a full factor.
\end{lemma}
\begin{proof} Let $(B_n)_{n\in\N}$ be an asymptotically invariant sequence for the $\F_3$-action
$\sigma_S$. Then  $(B_n)_{n\in\N}$ is an asymptotically invariant
sequence for the action restricted to the subgroup generated by
$\left\{T_a,T_b\right\}$. By construction, this is the $\F_2$-action
$\sigma$ described above, thus it is strongly ergodic by \cite[\S
4]{Schmidt80}.
  It follows that $(B_n)_{n\in\N}$ is trivial and then
$\sigma_S$ is strongly ergodic. Since $\F_3$ is not inner amenable,
the result now follows from Theorem \ref{TheoChoda}.
\end{proof}

In order to prove Theorem \ref{mainthm} we require the following
Theorem of Popa (\cite[Theorem 5.1]{popa06}):

\begin{theorem}\label{PopaMcDuff} If  $M_1$ and $M_2$ are full type $\II_1$ factors such that
$M_1\otimes R$, is isomorphic to $M_2\otimes R$ then there exists
 $t\in \R_{>0}$ such that
$M_1$ is isomorphic to $M_2^t$.
\end{theorem}
\begin{remark} By interchanging the roles of $M_1$ and $M_2$ one can
assume that $t\in (0,1]$. In which case $M_2^t$ is by definition the
type $\II_1$ factor $pM_2p$ where $p\in\mathcal P(M_2)$ is any
projection of trace equal to $t$ in $M_2$.
\end{remark}

\begin{theorem} The assignment $M_S\to M_S\otimes R$ is a Borel
reduction of $\simeq^{\mathscr F_{\II_1}}$ to isomorphism of McDuff
factors. 
\end{theorem}

\begin{proof}
It is fairly straightforward to prove that the map $M_S\to M_S\otimes
R$ is a Borel assignment (see for instance \cite[Corollary 3.8]{haagwin1}) . 
We are left to show that if $M_S\otimes
R$ is isomorphic to $M_{S'}\otimes
R$, then $M_{S}$ is isomorphic to $M_{S'}$.

Let us fix $S, S'\in\Ext(\sigma)$. Lemma \ref{lemmaFull} shows that $M_S$ and
$M_{S'}$  are full factors. By Theorem \ref{PopaMcDuff},
$M_{S}\otimes R$ is isomorphic to $M_{S'}\otimes R$ if and only if
there exists $t>0$ such that $M_{S'}$ is isomorphic to $(M_{S})^t$.
The proof is over once we show that $t=1$.\\
For this we make use of the celebrated theorem of Popa on $\II_1$ factors with trivial fundamental group \cite{popa}, \cite{popaPNAS}, (see also Connes's account in the Bourbaki S\'eminaire \cite{MR2111648}).
Indeed by \cite[Proposition]{popaPNAS}, there
exists a projection $p\in\mathcal P(L^\infty(\T^2))$, $\tau(p)=t$, such that the
inclusion of von Neumann algebras $(L^\infty(\T^2)\subset
L^\infty(\T^2)\rtimes_{\sigma_{S'}}\F_3)$ is isomorphic to the inclusion 
of von Neumann algebras $pL^\infty(\T^2)\subset p \left (L^\infty(\T^2)\rtimes_{\sigma_S}\F_3 \right ) p$. 
Feldman-Moore's
Theorem \cite{MR0578730} applies to conclude that the action $\sigma_S$ is stable orbit equivalent to
the action $\sigma_{S'}$, with compression constant $c=t$. Since $\F_3$ has non
trivial Atiyah's $\ell^2$-betti numbers, Gaboriau's Theorem on $\ell^2$-betti numbers for orbit equivalence relations \cite[Theorem 3.12]{GaboriauBetti} then implies that
$t=1$.
\end{proof}

\begin{proof}[Proof of Theorem  \ref{mainthm}]
Since $M_S\to M_S\otimes R$ is a Borel
reduction of $\simeq^{\mathscr F_{\II_1}}$ to isomorphism of McDuff
factors and the equivalence relation $\simeq^{\mathscr F_{\II_1}}$ is not classifiable by countable structures, 
it follows that the equivalence relation of isomorphism of McDuff factors is not classifiable by countable structures.
\end{proof}

\bibliographystyle{amsplain}
\bibliography{Gamma}

\end{document}